\documentclass[11pt]{article}

\usepackage[a4paper,margin=1.12in]{geometry}
\usepackage[T1]{fontenc}
\usepackage[utf8]{inputenc}
\usepackage{lmodern}
\usepackage{microtype}
\usepackage{amsmath,amssymb,amsthm,mathtools}
\usepackage{booktabs}
\usepackage{enumitem}
\usepackage{hyperref}

\hypersetup{colorlinks=true,linkcolor=blue,citecolor=blue,urlcolor=blue}
\setlist{nosep}
\allowdisplaybreaks

\newtheorem{theorem}{Theorem}[section]
\newtheorem{proposition}[theorem]{Proposition}
\newtheorem{lemma}[theorem]{Lemma}
\newtheorem{corollary}[theorem]{Corollary}

\theoremstyle{definition}

\theoremstyle{remark}

\newcommand{\E}{\mathbb{E}}
\newcommand{\Pp}{\mathbb{P}}
\newcommand{\1}{\mathbf{1}}
\newcommand{\Lcal}{\mathcal{L}}

\newcommand{\dens}{\operatorname{dens}}

\newcommand{\Bohr}{\operatorname{Bohr}}

\newcommand{\eps}{\varepsilon}

\title{Branch-Sensitive Density Increments for $k$-Configurations}
\author{Yao Zhi}

\begin{document}
\maketitle

\begin{abstract}
We refine the density-increment argument for $k$-configurations developed by Beker.  The point is not to introduce a new inverse theorem, but to retain information that is lost when the graph-counting output is compressed into a single uniform alternative.  For the transitive $K_k$ appearing in the $k$-configuration problem we keep a fixed moment throughout the graph recursion and obtain three genuinely different outputs: a row-degree anomaly, a non-centred grid anomaly, and a centred grid anomaly.  The centred case is treated directly through positivity of the associated Gram moments, avoiding the usual unbalancing step.  Each output is then routed through the appropriate local Kelley--Meka mechanism with its own density gain, rank cost, and width cost.  A weighted iteration budget prevents the most expensive local branches from being charged the maximal number of times.

For a finite abelian group $G$ of odd order and $A\subseteq G$ of density $\alpha$, write $L=\log(2/\alpha)$ and $\Lambda=L+\log(2k)$.  We obtain a lower bound of the form
\[
\Pp_{x_1,\ldots,x_k\in G}\!\left(\frac{x_i+x_j}{2}\in A\ \text{for all }1\le i\le j\le k\right)
\ge \exp\!\bigl(-C\Phi(k,L,\Lambda)\bigr),
\]
where
\[
\Phi(k,L,\Lambda)=k^{36}L^6\Lambda^6+k^{34}L^8\Lambda^5+k^{32}L^{14}\Lambda+k^{30}L^{16}.
\]
Consequently, in the Erd\H{o}s--Moser application one may replace the exponent $68$ in Beker's $k$-configuration route by $36$; in particular the same reduction yields the sum-free lower bound with every exponent $c<1/36$.
\end{abstract}

\medskip
\noindent\textbf{2020 Mathematics Subject Classification.} Primary 11B30; Secondary 11B75.

\noindent\textbf{Keywords.} arithmetic combinatorics; density increments; $k$-configurations; Erd\H{o}s--Moser problem; Bohr sets; Kelley--Meka method.

\section{Introduction and main results}

\subsection{The Erd\H{o}s--Moser problem and \texorpdfstring{$k$}{k}-configurations}

For finite sets $B\subseteq A\subseteq\mathbb R$, say that $B$ is \emph{sum-free with respect to $A$} if
\[
\{b+b':b,b'\in B,\ b\ne b'\}\cap A=\varnothing.
\]
The problem traditionally attributed to Erd\H{o}s and Moser asks for the largest function $f(n)$ such that every $n$-element set of integers contains a subset of size at least $f(n)$ which is sum-free with respect to the ambient set; see Erd\H{o}s' survey \cite{Erdos1965} and the subsequent work of Sudakov--Szemer\'edi--Vu \cite{SSV2005}.  Writing $\phi(A)$ for the largest size of such a subset of $A$ and $\phi(n)=\min_{|A|=n}\phi(A)$, the elementary scale is logarithmic.  Ruzsa constructed examples showing $\phi(n)\le \exp(O(\sqrt{\log n}))$ \cite{Ruzsa2005}, while Sudakov--Szemer\'edi--Vu proved the first superlogarithmic lower bound $\phi(n)\ge g(n)\log n$ with $g(n)\to\infty$ \cite{SSV2005}.  Sanders later established the substantially stronger form
\[
\phi(n)\ge (\log n)^{1+c}
\]
for some absolute $c>0$ \cite{Sanders2021}.  His argument uses a structural iteration in additive combinatorics rather than a direct quantitative theorem for the linear configurations considered below.

A second line of work studies the linear pattern that enters one route to this sum-free problem.  A \emph{$k$-configuration} in an abelian group in which multiplication by $2$ is invertible is generated by elements $x_1,\ldots,x_k$ together with all pairwise arithmetic means $(x_i+x_j)/2$; it is non-degenerate when the generators are pairwise distinct.  For comparison of quantitative results, let $F(\alpha,k)$ denote the least $N_0$ such that, for every $N\ge N_0$, every set $A\subseteq[N]$ of density at least $\alpha$ contains a non-degenerate $k$-configuration.  Dousse developed a Roth-type density-increment argument for these configurations and obtained a density threshold of iterated-logarithmic type \cite{Dousse2013}.  Shao, exploiting the fact that the defining system has complexity one and working efficiently inside Bohr sets, improved this to a logarithmic-density threshold: for each fixed $k$, a set $A\subseteq[N]$ of density $\alpha\gg (\log N)^{-c(k)}$ contains a non-degenerate $k$-configuration \cite{Shao2015}.  In the quantitative form relevant here this corresponds to a bound of the shape
\[
N\ge \exp\!\bigl((2/\alpha)^{O(k^2)}\bigr),
\]
as recorded and sharpened in Beker's work \cite{Beker2025}.

The analytic landscape changed with the Kelley--Meka breakthrough for three-term progressions \cite{KelleyMeka2023} and its Bohr-set exposition by Bloom--Sisask \cite{BloomSisask2023}.  Filmus--Hatami--Hosseini--Kelman extended the underlying sparse-counting and density-increment technology to binary systems of linear forms \cite{FHHK2024}.  Since the equations defining a $k$-configuration form such a binary system, Beker combined these tools with the Erd\H{o}s--Moser structural reduction and proved that a subset of $[N]$ of density $\alpha$ contains a non-degenerate $k$-configuration once
\[
N\ge \exp\!\bigl(Ck^{68}\log(2/\alpha)^{16}\bigr).
\]
Equivalently,
\[
F(\alpha,k)\le \exp\!\bigl(Ck^{68}\Lcal(\alpha)^{16}\bigr).
\]
He consequently obtained the explicit Erd\H{o}s--Moser exponent $c<1/68$ along this $k$-configuration route \cite[Theorems~1.1 and~1.2]{Beker2025}.  This should be distinguished from Sanders' earlier existence of an unspecified absolute exponent: Beker's contribution is a new route with explicit quantitative dependence on the configuration parameter.

\subsection{The branch-sensitive refinement}

The quantitative question addressed in this paper is whether Beker's argument loses avoidable information before the final iteration.  No new inverse theorem is introduced.  Instead, we retain the different outcomes of the sparse graph-counting recursion and charge each outcome according to its actual density gain and Bohr-set cost.  Two losses in Beker's implementation are especially relevant.  First, structurally different graph anomalies are converted to a common worst-case alternative.  Second, once a local increment is obtained, a uniform iteration bound effectively charges every step at the weakest density gain.  These two worst cases need not occur simultaneously.

For the transitive $K_k$ associated with a $k$-configuration, we run the graph recursion at a fixed moment and stop before the final unbalancing step.  This produces three outputs: a row-degree anomaly, a non-centred grid anomaly, and a centred grid anomaly.  The centred output is handled directly by positivity of Gram moments, while the non-centred output is routed through sparse rectangle extraction.  The three branches then enter the appropriate local Kelley--Meka backends with different density gains, rank costs, and width costs.  A weighted budget based on the sum of the actual gains keeps these costs separated through the global iteration.  The resulting improvement is therefore a refinement of Beker's density-increment architecture rather than a replacement of its analytic inputs.

Throughout, for $u\in(0,1]$ we write $\Lcal(u)=\log(2/u)$.  Given $\alpha\in(0,1]$ and $k\ge2$, we abbreviate
\[
L=\Lcal(\alpha),\qquad \Lambda=L+\log(2k).
\]
Our finite-group result is as follows.

\begin{theorem}[Finite-group configuration bound]\label{thm:finite-group}
There is an absolute constant $C>0$ with the following property.  Let $G$ be a finite abelian group of odd order, let $k\ge2$, and let $A\subseteq G$ have density $\alpha>0$.  Then
\begin{equation}\label{eq:main-count}
\Pp_{x_1,\ldots,x_k\in G}\!\left(\frac{x_i+x_j}{2}\in A\ \text{for all }1\le i\le j\le k\right)
\ge \exp\!\bigl(-C\Phi(k,L,\Lambda)\bigr),
\end{equation}
where
\begin{equation}\label{eq:Phi}
\Phi(k,L,\Lambda)
= k^{36}L^6\Lambda^6+k^{34}L^8\Lambda^5+k^{32}L^{14}\Lambda+k^{30}L^{16}.
\end{equation}
In particular, if $A$ contains no non-degenerate $k$-configuration, then
\[
|G|\le \exp\!\bigl(C\Phi(k,L,\Lambda)\bigr).
\]
\end{theorem}

The passage from finite groups to intervals is the standard cyclic embedding used in \cite{Beker2025}.  It gives the following consequence.

\begin{corollary}[Integer version]\label{cor:interval}
There is an absolute constant $C>0$ such that the following holds.  If $A\subseteq[N]$ has density at least $\alpha$ and
\[
N\ge \exp\!\bigl(C\Phi(k,\Lcal(\alpha),\Lcal(\alpha)+\log(2k))\bigr),
\]
then $A$ contains a non-degenerate $k$-configuration.
\end{corollary}

In terms of the extremal threshold introduced above, Corollary~\ref{cor:interval} gives
\[
F(\alpha,k)\le
\exp\!\bigl(C\Phi(k,\Lcal(\alpha),\Lcal(\alpha)+\log(2k))\bigr).
\]
Thus the leading dependence on $k$ in Beker's bound is reduced from $k^{68}$ to $k^{36}$, while the full statement retains the lower-order mixed terms displayed in \eqref{eq:Phi}.

Finally, combining Theorem~\ref{thm:finite-group} with the Erd\H{o}s--Moser reduction used by Beker and Sanders gives the advertised application.

\begin{theorem}[Erd\H{o}s--Moser consequence]\label{thm:erdos-moser}
Let $c<1/36$.  For every sufficiently large finite non-empty set $A\subseteq\mathbb Z$, there is a subset $B\subseteq A$ of size at least
\[
(\log |A|)^{1+c}
\]
such that $b_1+b_2\notin A$ whenever $b_1,b_2\in B$ are distinct.
\end{theorem}

The proof has two steps.  Section~\ref{sec:local} replaces the uniform graph-counting output in Beker's local proposition by branch-sensitive alternatives.  The transitive $K_k$ recursion is run at a fixed moment $P\asymp k^2L$.  A row-degree anomaly has strength $\asymp k^{-2}$; a non-centred grid anomaly has strength $\asymp k^{-1}$; a centred grid anomaly has strength $\asymp k^{-2}$ in grid norm and, after a direct Gram argument, yields a density effect of strength $\asymp k^{-4}$.  The corresponding local density increments have substantially different costs.  Section~\ref{sec:iteration} uses the elementary budget
\[
\sum_s g_s\ll L,
\]
where $g_s$ is the relative density gain at step $s$, to sum the costs branch by branch.  The rank bottleneck comes from the centred branch, while the width bottleneck comes from the slow local-uniformisation branch; these two worst cases cannot both occur at the maximal number of iterations.

\section{The local configuration-or-increment proposition}\label{sec:local}

We follow the normalisations of \cite{Beker2025}.  Throughout the finite-group argument, $G$ carries the uniform probability measure $\mu$.  For a non-empty set $S\subseteq G$, write $\mu_S=\mu(S)^{-1}\1_S$.  For functions $f,g\colon G\to\mathbb C$, their difference convolution is
\[
(f\circ g)(x)=\E_{y\in G}f(y)\overline{g(y-x)}.
\]
If $\nu$ is a probability density on $G$, then $\|h\|_{L^p(\nu)}=(\E_x|h(x)|^p\nu(x))^{1/p}$.  We write $[k]=\{1,\ldots,k\}$ and $\psi_2(x)=2x$.  For $\Gamma\subseteq\widehat G$ and $\sigma\in\{\pm1\}$, put
\[
\Gamma\psi_2^\sigma=\{\xi\circ\psi_2^\sigma:\xi\in\Gamma\}.
\]
We use Beker's Bohr-set convention
\[
\Bohr(\Gamma;\rho)=\{x\in G:|\xi(x)-1|\le\rho\ \text{for every }\xi\in\Gamma\},
\qquad
B_\lambda=\Bohr(\Gamma;\lambda\rho)
\]
when $B=\Bohr(\Gamma;\rho)$; regularity is as in \cite[Appendix~B]{Beker2025}.

If $X$ is a finite non-empty set, expectations over $X$ are uniform.  For $f\colon X\times Y\to\mathbb R$ and positive integers $r,p$, set
\[
\|f\|_{U(r,p)}
=\left|\E_{x_1,\ldots,x_r\in X}\E_{y_1,\ldots,y_p\in Y}
\prod_{a=1}^r\prod_{b=1}^p f(x_a,y_b)\right|^{1/(rp)}.
\]
The graph-counting recursion below is the one used by Beker, itself based on Filmus--Hatami--Hosseini--Kelman \cite{Beker2025,FHHK2024}.  We reproduce the quantitative part that changes.

\subsection{A fixed-moment pre-unbalancing trichotomy}

Let $m=\binom{k}{2}$.  We shall use an absolute constant $C_0$ chosen sufficiently large and put
\begin{equation}\label{eq:P-eta-theta}
P=4\left\lceil C_0m\Lcal(\alpha_*)\right\rceil,
\qquad \theta=\frac1{100k},
\qquad \eta=\frac1{16m}.
\end{equation}
The factor $4$ ensures that $P$ is divisible by $4$.

\begin{lemma}[Fixed-moment compatibility]\label{lem:fixed-moment-compatibility}
Consider a recursive node $H$ in the proof of Beker's graph-counting lemma \cite[Lemma~2.2]{Beker2025}, and suppose that every edge function at that node has mean at least $\alpha_*$.  Write
\[
p_H=4\left\lceil |E(H)|\Lcal(\alpha_*)\right\rceil,
\qquad q_H=\frac{p_H}{p_H-1}.
\]
Let $P$ be a multiple of $4$ with $P\ge p_H$, and put $q_P=P/(P-1)$.  Then the analytic part of Beker's recursive step up to, but not including, the final unbalancing argument may be carried out with the common moment $P$ in place of $p_H$, with the same absolute constants.  In particular, the degree estimate in the one-in-neighbour case and the centred/non-centred product estimate in the multiple-in-neighbour case remain valid with every grid moment equal to $P$.
\end{lemma}

\begin{proof}
The node-dependent moment in the preliminary H\"older estimate of \cite[Lemma~2.2]{Beker2025} is used through an upper $L^{q_H}$ bound, where $q_H$ is conjugate to $p_H$.  Since $P\ge p_H$, we have $q_P\le q_H$.  All spaces involved are probability spaces, and hence monotonicity of $L^q$ norms gives
\[
\|g\|_{q_P}\le \|g\|_{q_H}
\]
for every function $g$ occurring in that estimate.  Thus the same preliminary bound is available with $q_P$ in place of $q_H$.

After this preliminary estimate, Beker's recursive step consists of H\"older, Cauchy--Schwarz, and Gowers--Cauchy--Schwarz inequalities applied at the chosen grid moment.  These inequalities are valid at the common moment $P$; the divisibility by $4$ guarantees the parity required in the repeated Cauchy--Schwarz steps.  Deleting a sink or replacing an incoming edge by a copied edge does not change the mean of any surviving or copied edge function, so the lower bound $\alpha_*$ persists throughout the recursion.  Therefore replacing the node-dependent $p_H$ by the larger common value $P$ does not weaken any estimate before the final unbalancing step.  The two pre-unbalancing estimates used below consequently hold with $P$, as claimed.
\end{proof}

\begin{proposition}[Fixed-moment trichotomy]\label{prop:fixed-moment}
Let $X_1,\ldots,X_k$ be finite non-empty sets and, for every $1\le i<j\le k$, let
$f_{ij}\colon X_i\times X_j\to\{0,1\}$ have mean $\rho_{ij}\ge\alpha_*>0$.  Suppose
\[
\left|
\E_{x_1,\ldots,x_k}\prod_{i<j} f_{ij}(x_i,x_j)
-
\prod_{i<j}\rho_{ij}
\right|
\ge \frac12\prod_{i<j}\rho_{ij}.
\]
Then for some edge $e=(u,v)$ at least one of the following holds.
\begin{enumerate}[label=\textup{(\Alph*)}]
\item\label{it:D} With $d_e(x)=\E_y f_e(x,y)$,
\[
\|d_e-\rho_e\|_P\ge \eta\rho_e.
\]
\item\label{it:G} There is an even integer $r$ with $2\le r\le 2(k-1)$ such that
\[
\|f_e\|_{U(r,P)}\ge (1+\theta)\rho_e.
\]
\item\label{it:C} If $F_e=f_e/\rho_e$, $R_e(x)=\E_yF_e(x,y)$ and $J_e=F_e-R_e$, then
\[
\|J_e\|_{U(2,P)}\ge \eta.
\]
\end{enumerate}
\end{proposition}

\begin{proof}
If $k=2$, the graph has a single edge and its graph count is exactly the mean of that edge, so the discrepancy hypothesis is impossible.  We may therefore assume $k\ge3$.  For an oriented graph $H$, let $d_1(H)$ denote the number of vertices of degree one in the underlying undirected graph.  We follow the induction on the parameter $\kappa(H)=2|E(H)|-d_1(H)$ in Beker's proof of his graph-counting lemma \cite[Lemma~2.2]{Beker2025}, but we retain the centred output instead of applying the final unbalancing step.  As in that proof, isolated vertices and isolated edges may first be removed: their contribution factors from the normalised graph count, so they cannot create the assumed relative discrepancy.  Thus the induction is non-trivial only when $\kappa(H)\ge2$.

We strengthen the induction by keeping the external bounds
\[
|E(H)|\le m,\qquad \max_w d^-(w)\le k-1.
\]
All descendants produced by the recursion satisfy these bounds: deleting a sink decreases the number of edges, while splitting one incoming edge to a new leaf preserves the number of edges and creates a vertex of indegree one.  In the split graph the new edge carries exactly the same function as the edge it replaces, so any output on that copied edge canonically corresponds to an edge of the original transitive $K_k$.

At a recursive node let $\eps_H$ denote the current relative counting discrepancy.  Beker takes
\[
\eps' = \frac{\kappa(H)-1}{\kappa(H)}\eps_H.
\]
If a sink has indegree $\ell$, deletion of the sink gives a graph $H'$ with
$\kappa(H')\le \kappa(H)-\ell$, while the split graph $H''$ satisfies
$\kappa(H'')\le \kappa(H)-1$.  Hence, in either recursive branch,
\[
\frac{\eps'}{\kappa(H_{\rm child})}
\ge \frac{\eps_H}{\kappa(H)}.
\]
Initially $\eps_{K_k}=1/2$ and $\kappa(K_k)=2m$, and therefore every node in the recursion satisfies
\begin{equation}\label{eq:eps-kappa-invariant}
\frac{\eps_H}{\kappa(H)}\ge \frac1{4m}.
\end{equation}

Every descendant has at most $m$ edges, so after enlarging the absolute constant $C_0$ if necessary, the value of $P$ in \eqref{eq:P-eta-theta} dominates the node-dependent moment $p_H$ at every stage.  Lemma~\ref{lem:fixed-moment-compatibility} therefore allows the entire pre-unbalancing recursion to be run with this single value of $P$.

Let $s$ be a sink with in-neighbours $u_1,\ldots,u_\ell$.  If $\ell=1$, the H\"older step in Beker's proof gives
\[
\left\|\E_yF_{u_1s}(\cdot,y)-1\right\|_P
\ge \frac{\eps_H}{3\kappa(H)}.
\]
By \eqref{eq:eps-kappa-invariant}, the right-hand side is at least $1/(12m)>\eta$, and \ref{it:D} follows.

Suppose $\ell\ge2$.  The Gowers--Cauchy--Schwarz calculation in the same proof gives
\begin{equation}\label{eq:pre-unbalancing-product}
\|J\|_{U(2,P)}
\prod_{i=2}^{\ell}\|F_{u_is}\|_{U(2(\ell-1),P)}
\ge \frac{\eps_H}{3\kappa(H)},
\end{equation}
where $J=F_{u_1s}-\E_yF_{u_1s}(\cdot,y)$.  If one of the non-centred factors is at least $1+\theta$, then \ref{it:G} holds with $r=2(\ell-1)\le2(k-1)$.  Otherwise
\[
\prod_{i=2}^{\ell}\|F_{u_is}\|_{U(2(\ell-1),P)}
\le (1+\theta)^{k-1}<e^{1/100}.
\]
Combining this with \eqref{eq:pre-unbalancing-product} and \eqref{eq:eps-kappa-invariant} yields
\[
\|J\|_{U(2,P)}
\ge \frac1{12e^{1/100}m}>\eta,
\]
which is \ref{it:C}.  This completes the induction.
\end{proof}

The improvement over the usual uniform graph-counting output is that the centred branch in Proposition~\ref{prop:fixed-moment} is not converted into a non-centred grid norm at moment $\asymp P/\eta^2$.  We instead use the Gram structure directly.

\subsection{The centred Gram branch}

\begin{lemma}[Positive Gram tail]\label{lem:gram-tail}
Let $f\colon X\times Y\to\{0,1\}$ have mean $\rho>0$.  Set
\[
F=f/\rho,\qquad R(x)=\E_yF(x,y),\qquad J=F-R
\]
and
\[
K(x,x')=\E_yJ(x,y)J(x',y).
\]
Assume that $P$ is even, $\|J\|_{U(2,P)}\ge\eta$, and
\begin{equation}\label{eq:gram-size-condition}
2^P\ge \frac{4\rho^{-2}}{\eta^2}.
\end{equation}
Then
\[
\|(K)_+\|_P\ge \eta^2/2.
\]
Consequently, with $\tau=\eta^2/4$ and
$E_+=\{(x,x'):K(x,x')\ge\tau\}$,
\begin{equation}\label{eq:Eplus-mass}
\Pp(E_+)\ge \left(\frac{\eta^2\rho^2}{4}\right)^P.
\end{equation}
\end{lemma}

\begin{proof}
For every positive integer $q$,
\begin{equation}\label{eq:gram-moments}
\E_{x,x'}K(x,x')^q
=
\E_{y_1,\ldots,y_q}
\left(\E_x\prod_{a=1}^qJ(x,y_a)\right)^2
\ge0.
\end{equation}
Moreover,
\[
\E K^P=\|J\|_{U(2,P)}^{2P}\ge\eta^{2P}.
\]
Since $|J|\le\rho^{-1}$, we have $|K|\le M:=\rho^{-2}$.  Write
$A_+=\E(K_+)^P$ and $A_-=\E(K_-)^P$.  If $A_+\ge\eta^{2P}/2$, the first conclusion is immediate.  Otherwise $A_-\ge\eta^{2P}/2$.  Since $P+1$ is odd, \eqref{eq:gram-moments} gives
\[
\E K_+^{P+1}\ge \E K_-^{P+1}.
\]
By monotonicity of $L^q$ norms and the bound $K_+\le M$,
\[
M A_+\ge \E K_+^{P+1}\ge A_-^{1+1/P}.
\]
Condition \eqref{eq:gram-size-condition} now implies $A_+^{1/P}\ge\eta^2/2$.

For the second statement,
\[
\E K_+^P
\le \tau^P+M^P\Pp(E_+),
\]
whereas the first part gives $\E K_+^P\ge(\eta^2/2)^P$.  Since $P\ge1$,
$(\eta^2/2)^P-(\eta^2/4)^P\ge(\eta^2/4)^P$, and \eqref{eq:Eplus-mass} follows.
\end{proof}

In the application, $\rho\ge\alpha_*$ and $P\asymp m\Lcal(\alpha_*)$.  Taking $C_0$ sufficiently large in \eqref{eq:P-eta-theta} guarantees \eqref{eq:gram-size-condition} uniformly.

\begin{lemma}[Projection of a positive Gram tail]\label{lem:gram-projection}
Under the hypotheses of Lemma~\ref{lem:gram-tail}, let $e$ denote the right-hand side of \eqref{eq:Eplus-mass}.  Then at least one of the following holds.
\begin{enumerate}[label=\textup{(\roman*)}]
\item There is a set $H\subseteq X$ of measure at least $e/2$ such that
$\E_y f(x,y)>2\rho$ for every $x\in H$.
\item There are sets $C\subseteq X$ and $T\subseteq Y$ with
\[
\Pp_X(C)\ge e/2,\qquad \Pp_Y(T)\ge \tau\rho^2,
\]
and
\[
\E_{(x,y)\in C\times T}f(x,y)\ge(1+\tau/4)\rho.
\]
\item There is a set $S\subseteq X$ with $\Pp_X(S)\ge e\tau/16$ such that
\[
\E_y f(x,y)\le(1-\tau/8)\rho
\quad\text{for all }x\in S.
\]
\end{enumerate}
\end{lemma}

\begin{proof}
Write
\[
d(x)=\E_yf(x,y),\qquad d_2(x,x')=\E_yf(x,y)f(x',y).
\]
Since $J(x,y)=(f(x,y)-d(x))/\rho$,
\begin{equation}\label{eq:K-codegree}
K(x,x')=\frac{d_2(x,x')-d(x)d(x')}{\rho^2}.
\end{equation}
Let $H=\{x:d(x)>2\rho\}$.  If
$\Pp(E_+\cap(H\times X))\ge e/2$, then $\Pp_X(H)\ge e/2$ and the first alternative holds.

Otherwise, by averaging there is $x_0\notin H$ for which
\[
C=\{x:K(x_0,x)\ge\tau\}
\]
has measure at least $e/2$.  Put $d_0=d(x_0)$ and $T=\{y:f(x_0,y)=1\}$.  For $x\in C$, \eqref{eq:K-codegree} yields
\[
d_2(x_0,x)\ge d_0d(x)+\tau\rho^2.
\]
Since $d_2(x_0,x)\le d_0$, we have $d_0\ge\tau\rho^2$, while $x_0\notin H$ gives $d_0\le2\rho$.  Hence
\[
\E_{C\times T}f
=\frac{\E_{x\in C}d_2(x_0,x)}{d_0}
\ge \E_Cd(x)+\frac{\tau\rho^2}{d_0}
\ge \E_Cd(x)+\frac\tau2\rho.
\]
If $\E_Cd\ge(1-\tau/4)\rho$, alternative (ii) follows.  Otherwise set
\[
S=\{x\in C:d(x)\le(1-\tau/8)\rho\}.
\]
If $\Pp_C(S)<\tau/8$, then
\[
\E_Cd>(1-\tau/8)^2\rho>(1-\tau/4)\rho,
\]
a contradiction.  Thus $\Pp_X(S)\ge(e/2)(\tau/8)=e\tau/16$, proving (iii).
\end{proof}

The ordinary row-degree branch also admits a direct level extraction.  We record it in the form needed below.

\begin{lemma}[Degree-level extraction]\label{lem:degree-level}
Let $f\colon X\times Y\to\{0,1\}$ have mean $\rho$, put $d(x)=\E_yf(x,y)$, and suppose
$\|d-\rho\|_P\ge\eta\rho$.  Then at least one of the sets
\[
S_+=\{x:d(x)\ge(1+\eta/4)\rho\},\qquad
S_-=\{x:d(x)\le(1-\eta/4)\rho\}
\]
satisfies
\[
\Pp_X(S_+)\ge \frac14(\eta\rho)^P
\quad\text{or}\quad
\Pp_X(S_-)\ge \frac14\eta^P,
\]
respectively.
\end{lemma}

\begin{proof}
Set $Z=d/\rho-1$.  Since $\|Z\|_P\ge\eta$, either $\E(Z_+)^P\ge\eta^P/2$ or $\E(Z_-)^P\ge\eta^P/2$.  In the first case, using $Z_+\le\rho^{-1}$ and splitting at level $\eta/4$ gives
\[
\frac{\eta^P}{2}
\le \E(Z_+)^P
\le \left(\frac{\eta}{4}\right)^P
   +\rho^{-P}\Pp_X(S_+).
\]
As $P\ge2$, the difference between the first two terms is at least $\eta^P/4$, and hence $\Pp_X(S_+)\ge(\eta\rho)^P/4$.  In the second case, $Z_-\le1$ and the same argument gives
\[
\frac{\eta^P}{2}
\le \left(\frac{\eta}{4}\right)^P+\Pp_X(S_-),
\]
so $\Pp_X(S_-)\ge\eta^P/4$.
\end{proof}

\subsection{Insertion into the local Bohr-set argument}

We now place the preceding graph outputs into Beker's local setup.  We use without reproving the standard facts about regular Bohr sets, the local correlation-to-increment estimate \cite[Corollary~3.5]{Beker2025}, and the asymmetric rectangle-to-increment estimate \cite[Proposition~3.4]{Beker2025}.  The latter is essential because it returns the density increment on a designated first endpoint.  We also use the bipartite sparse rectangle extraction of Filmus--Hatami--Hosseini--Kelman \cite[Corollary~2.3]{FHHK2024}.

The local preparation is exactly the one in \cite[proof of Proposition~3.2]{Beker2025}.  Let $B=\Bohr(\Gamma;\rho_B)$ be a regular Bohr set of rank $d$, and let $A\subseteq B$ have relative density $\alpha$.  Write $m=\binom{k}{2}$ and, as in Beker's application of his graph-counting theorem with discrepancy parameter $1/2$, take
\[
\delta=\frac{1}{64000m^2},\qquad
\gamma=2^{-10}\delta,\qquad
\lambda=\frac{c\alpha\gamma}{2dk}\asymp\frac{\alpha}{d}k^{-5},
\]
where $c>0$ is a sufficiently small absolute constant.  As in \cite[proof of Proposition~3.2]{Beker2025}, first put
\[
\widetilde B=B\cap(2\cdot B),
\qquad
\widetilde\Gamma=\Gamma\cup\{\xi\psi_2^{-1}:\xi\in\Gamma\}.
\]
Thus $\widetilde B=\Bohr(\widetilde\Gamma;\rho_B)$ has effective rank $\widetilde d:=|\widetilde\Gamma|\le2d$.  Using regular-dilate selection \cite[Lemma~B.10]{Beker2025}, one chooses parameters $\lambda_1,\ldots,\lambda_{k+1}$ with $\lambda_{k+1}=1$ and
\[
\frac{\lambda}{2}\lambda_{j+1}\le\lambda_j\le\lambda\lambda_{j+1}
\quad(j\in[k]),
\]
so that $B^{(j)}:=\widetilde B_{\lambda_j}$ is regular for every $j\in[k]$.  In particular, all the $B^{(j)}$ have the same frequency set $\widetilde\Gamma$, and for $i<j$ one has
\[
B^{(i)}=(B^{(j)})_{\lambda_i/\lambda_j},
\qquad
(\lambda/2)^{j-i}\le\lambda_i/\lambda_j\le\lambda.
\]
The local-uniformisation alternative is obtained from \cite[Lemma~3.4]{FHHK2024} exactly as in \cite[proof of Proposition~3.2]{Beker2025}; it either gives immediately a density increment of relative size $\asymp k^{-5}$ or else provides a translate $A'$ for which the relevant local densities all lie in $[(1-\gamma)\alpha,(1+\gamma)\alpha]$.  We call the former the $U$-branch and, in the latter case, set
\[
\alpha_0=(1-\gamma)\alpha.
\]
Thus every dense endpoint used below has relative density at least $\alpha_0\ge\alpha/2$ in the relevant $B^{(j)}$.

In the latter case define, as in \cite{Beker2025},
\[
A_j=A'\cap B^{(j)},\qquad
\widetilde A_j=A'\cap \tfrac12 B^{(j)},
\]
and for $i<j$ the bipartite graph
\[
f_{ij}(x,y)=\1_{2\widetilde A_j}(x+y),\qquad (x,y)\in A_i\times A_j.
\]
Let $\alpha_{ij}=\E f_{ij}$ and $\alpha_*=\min_{i<j}\alpha_{ij}$.  If some $|\alpha_{ij}-\alpha|\ge\delta\alpha/4$, Beker's pre-graph argument \cite[proof of Proposition~3.2]{Beker2025} produces a density increment of relative size $\asymp k^{-4}$, additional rank $O(k^{16}(\log(2k))^4L^6)$, and logarithmic width loss $O(k\Lcal(\alpha/d)+k\log k)$.  Since $(\log(2k))^4\ll k^4$ and $L^3\ll\Lambda^5$, this additional rank is $O(k^{20}L^3\Lambda^5)$; hence this case is absorbed by the $C$-budget.  Otherwise
\[
|\alpha_{ij}-\alpha|\le\frac{\delta}{4}\alpha
\quad(i<j),
\]
and hence $\alpha_*\ge\alpha/2$, so Proposition~\ref{prop:fixed-moment} applies.

The following elementary observation will be used to pass from graph densities to the normalised correlations required by Beker's local density-increment results.

\begin{lemma}[Local normalisation]\label{lem:local-normalisation}
Put $\beta_j=\mu_{B^{(j)}}(2\widetilde A_j)$, so that $|\beta_j-\alpha|\le\gamma\alpha$.  Let $S\subseteq A_i$ and let $\varnothing\ne T\subseteq A_j$.  Suppose $0<u\le1$ and $u\ge\delta+4\gamma$.
\begin{enumerate}[label=\textup{(\alph*)}]
\item If
\[
\E_{(x,y)\in S\times T}f_{ij}(x,y)\ge(1+u)\alpha_{ij},
\]
then
\[
\langle\mu_{2\widetilde A_j}\circ\mu_T,\mu_S\rangle
\ge(1+u/4)\mu(B^{(j)})^{-1}.
\]
\item If
\[
\E_{(x,y)\in S\times T}f_{ij}(x,y)\le(1-u)\alpha_{ij},
\]
then
\[
\langle\mu_{2\widetilde A_j}\circ\mu_T,\mu_S\rangle
\le(1-u/2)\mu(B^{(j)})^{-1}.
\]
\end{enumerate}
Moreover, if a density-increment backend produces a translate of a regular Bohr set on which either $2\widetilde A_j$ or $A_j$ has relative density at least $(1+v)$ times its density in $B^{(j)}$, where $0<v\le1$ and $v\ge4\gamma$, then the same translate has relative density at least $(1+v/2)\alpha$ in the corresponding translate of $A$.
\end{lemma}

\begin{proof}
By the definition of the normalised indicators and difference convolution,
\[
\langle\mu_{2\widetilde A_j}\circ\mu_T,\mu_S\rangle
=\frac{\E_{(x,y)\in S\times T}f_{ij}(x,y)}{\beta_j\mu(B^{(j)})}.
\]
Write $a=\delta/4$ and $b=\gamma$.  In the positive case, using $\alpha_{ij}\ge(1-a)\alpha$ and $\beta_j\le(1+b)\alpha$, we obtain
\[
\frac{(1+u)\alpha_{ij}}{\beta_j}
\ge\frac{(1+u)(1-a)}{1+b}
=1+\frac{u-a-ua-b}{1+b}.
\]
Since $u\le1$ and $u\ge4(a+b)=\delta+4\gamma$, the numerator in the last fraction is at least $u/2$; as $b<1$, this is at least $u/4$ after division by $1+b$.  This proves part~(a).  In the negative case, $\alpha_{ij}\le(1+a)\alpha$ and $\beta_j\ge(1-b)\alpha$, and hence
\[
1-\frac{(1-u)\alpha_{ij}}{\beta_j}
\ge\frac{u+ua-a-b}{1-b}
\ge\frac{u-(a+b)}{1-b}
\ge\frac{u}{2},
\]
which proves part~(b).

For the final assertion, each dense endpoint has initial density $\beta\ge(1-\gamma)\alpha$.  Thus
\[
(1+v)\beta\ge(1+v)(1-\gamma)\alpha
\ge(1+v/2)\alpha
\]
whenever $v\ge4\gamma$, since $v/2-\gamma-v\gamma\ge v/2-2\gamma\ge0$.  The same statement for $A_j$ follows from its density bound in the local-uniformisation alternative.  This is the normalisation calculation used in \cite[proof of Proposition~3.2, equations (6)--(8)]{Beker2025}.
\end{proof}

\begin{proposition}[Branch-sensitive local alternative]\label{prop:local}
There are absolute constants $c,C>0$ such that the following holds.  Let $G$ be a finite abelian group of odd order, let $B=\Bohr(\Gamma;\rho_B)$ be a regular Bohr set of rank $d$, and let $A\subseteq B$ have relative density $\alpha>0$.  Put $L=\Lcal(\alpha)$ and $\Lambda=L+\log(2k)$.  Then either
\begin{enumerate}[label=\textup{(\roman*)}]
\item\label{local-count} the proportion of $k$-tuples generating a $k$-configuration in $A$ is at least
\[
\exp\!\bigl(-C(k^2\log k)d\Lcal(\alpha/d)\bigr)\rho_B^{2kd};
\]
or there is a translate of $A$ having increased density on a regular Bohr set $B'=\Bohr(\Gamma';\rho_B')$ and a sign $\sigma\in\{\pm1\}$ such that the step belongs to one of the following cost classes:
\item $U$-type:
\[
\alpha'\ge(1+ck^{-5})\alpha,
\qquad \Gamma'\subseteq\Gamma\cup\{\gamma\psi_2^{\sigma}:\gamma\in\Gamma\},
\]
with no new auxiliary frequencies and logarithmic width loss
$O(k\Lcal(\alpha/d)+k\log k)$;
\item $D$-type:
\[
\alpha'\ge(1+ck^{-2})\alpha,
\qquad |\Gamma'\setminus(\Gamma\cup\Gamma\psi_2^\sigma)|
\le Ck^{12}L^4\Lambda^4;
\]
\item $C$-type:
\[
\alpha'\ge(1+ck^{-4})\alpha,
\qquad |\Gamma'\setminus(\Gamma\cup\Gamma\psi_2^\sigma)|
\le Ck^{20}L^3\Lambda^5;
\]
\item $G$-type:
\[
\alpha'\ge(1+ck^{-1})\alpha,
\qquad |\Gamma'\setminus(\Gamma\cup\Gamma\psi_2^\sigma)|
\le Ck^{19}L^{11}.
\]
\end{enumerate}
For the $D$- and $C$-type branches the logarithmic width loss is
$O(k\Lcal(\alpha/d)+k\log k)$, while for the $G$-type branch it is
$O(k\Lcal(\alpha/d)+k^3L^2)$.
\end{proposition}

\begin{proof}
The count alternative is unchanged from Beker's proof of \cite[Proposition~3.2]{Beker2025}: once the graph count is at least one half of the product of the edge densities, the product of the nested Bohr-set measures, together with the standard Bohr-size estimate \cite[Lemma~B.7]{Beker2025}, gives the lower bound in \ref{local-count}.  We therefore treat the non-counting branches.

The $U$-branch is the local-uniformisation failure in \cite[proof of Proposition~3.2]{Beker2025}.  Since $\gamma\asymp k^{-4}$, the relative density gain is $\gamma/(4k)\asymp k^{-5}$.  This step only closes the old frequency set under one multiplication-by-$2$ pullback and introduces no auxiliary set of frequencies.

Suppose next that Proposition~\ref{prop:fixed-moment}\ref{it:D} holds for some edge $(i,j)$.  Lemma~\ref{lem:degree-level}, together with $P=O(k^2L)$, $\eta\asymp k^{-2}$, and $\rho_e=\alpha_{ij}\asymp\alpha$, gives a positive- or negative-degree carrier $S\subseteq A_i$ of Bohr-relative logarithmic density
\begin{equation}\label{eq:HD}
H_D=O(k^2L\Lambda).
\end{equation}
On $S$ the row density differs from $\alpha_{ij}$ by a relative amount at least $\eta/4$.  Since $\eta/4\gg\delta+4\gamma$, Lemma~\ref{lem:local-normalisation}, with $T=A_j$, converts this into
\[
\left|\langle\mu_{2\widetilde A_j}\circ\mu_{A_j},\mu_S\rangle
      -\mu(B^{(j)})^{-1}\right|
\ge c\eta\,\mu(B^{(j)})^{-1}.
\]
Thus the correlation defect has size $\eps_D\asymp\eta\asymp k^{-2}$.  The set $S$ lies in the regular dilate $B^{(i)}=(B^{(j)})_{\lambda_i/\lambda_j}$.  Since $\widetilde d\le2d$, $\alpha_0\ge\alpha/2$, and $\lambda\asymp(\alpha/d)k^{-5}$, after choosing the absolute constant in $\lambda$ sufficiently small we have
\[
\lambda_i/\lambda_j\le\lambda\le c\eps_D\alpha_0/\widetilde d.
\]
Consequently all hypotheses of \cite[Corollary~3.5]{Beker2025} are satisfied with ambient set $B^{(j)}$, carrier set $S$, endpoint-density parameter $\alpha_0$, and carrier dilate $B^{(i)}$.  Substituting $\eps_D\asymp k^{-2}$ and \eqref{eq:HD} into its rank bound gives
\[
\eps_D^{-4}\Lcal(\min(\eps_D,e^{-H_D}))^2
\Lcal(\eps_D\alpha_0)^2\Lcal(\alpha_0)^2
\ll k^{12}L^4\Lambda^4.
\]
The corollary returns an increment by a relative factor $1+c\eps_D$ on one of the two dense endpoints $2\widetilde A_j$ and $A_j$.  Since $\eps_D\gg\gamma$, the final assertion of Lemma~\ref{lem:local-normalisation} converts this into a density at least $(1+c'\eps_D)\alpha$ for a suitable absolute $c'>0$.  If the increment is on $A_j$, it is already an increment for a translate of $A$.  If it is on $2\widetilde A_j$, apply the automorphism $x\mapsto x/2$.  The pullback of a regular Bohr set $\Bohr(\Gamma';\rho')$ is $\Bohr(\Gamma'\psi_2;\rho')$; hence rank, width, and regularity are unchanged.  Since $\widetilde A_j\subseteq A'$, this again gives an increment for a translate of $A$.  This is the final rescaling step used in \cite[proof of Proposition~3.2]{Beker2025}.  We obtain the $D$-type alternative.

Assume Proposition~\ref{prop:fixed-moment}\ref{it:G}.  Since $rP=O(k^3L)$ and
\[
\|f_e\|_{U(r,P)}^{rP}
\ge(1+\theta)^{rP}\rho_e^{rP},
\]
Corollary~2.3 of \cite{FHHK2024}, applied to $K_{r,P}$, produces a rectangle $S\times T\subseteq A_i\times A_j$ on which the edge density is at least $(1+\theta/2)\rho_e$ and
\[
\dens_{A_i}(S)\dens_{A_j}(T)\gg \theta\alpha_*^{rP+1}.
\]
As both densities are at most one, each is at least the product on the left.  After multiplying by the ambient densities of $A_i$ and $A_j$, both $S$ and $T$ therefore have Bohr-relative logarithmic density $O(k^3L^2)$.  Since $\theta/2\gg\delta+4\gamma$, Lemma~\ref{lem:local-normalisation} gives
\[
\langle\mu_{2\widetilde A_j}\circ\mu_T,\mu_S\rangle
\ge(1+c\theta)\mu(B^{(j)})^{-1}.
\]
Write $\gamma_S=\mu_{B^{(i)}}(S)$.  H\"older's inequality yields, for every $p>1$,
\[
\|\mu_{2\widetilde A_j}\circ\mu_T\|_{L^p(\mu_{B^{(i)}})}
\ge (1+c\theta)\gamma_S^{1/p}\mu(B^{(j)})^{-1}.
\]
Choose an integer
\[
p_G=\left\lceil C_1\theta^{-1}
\bigl(\Lcal(\gamma_S)+\Lcal(\theta)\bigr)\right\rceil
=O(k^4L^2),
\]
where $C_1$ is a sufficiently large absolute constant.  Then $\gamma_S^{1/p_G}\ge1-c\theta/8$.  The additional term $\theta^{-1}\Lcal(\theta)$ ensures, after adjusting $C_1$, that
$p_G\ge C\eps_G^{-1}\Lcal(\eps_G)$ for every $\eps_G\asymp\theta$, which is the moment hypothesis in \cite[Proposition~3.4]{Beker2025}.

Using regular-dilate selection \cite[Lemma~B.10]{Beker2025}, choose regular Bohr sets
\[
B'=B^{(i)}_{\lambda'},
\qquad
B''=(B')_{\lambda''},
\qquad
\lambda',\lambda''\in[c/(2\widetilde d),c/\widetilde d].
\]
They have rank $\widetilde d$ and satisfy $B''\subseteq(B')_{c/\widetilde d}$.  Beker's local averaging comparison \cite[Lemma~B.11]{Beker2025}, followed by the averaging lemma \cite[Lemma~A.2]{Beker2025}, transfers the preceding $L^{p_G}(\mu_{B^{(i)}})$ lower bound to an $L^{p_G}$ norm with respect to a measure of the form $\mu_{B'}*\mu_{B''+x}$, with loss at most $2^{-1/p_G}=1-O(\theta)$.  Taking $C_1$ large and reducing the absolute constant in the anomaly, we obtain
\[
\|\mu_{2\widetilde A_j}\circ\mu_T\|_{L^{p_G}(\mu_{B'}*\mu_{B''+x})}
\ge(1+\eps_G)\mu(B^{(j)})^{-1},
\qquad \eps_G\asymp\theta.
\]
Thus every hypothesis of \cite[Proposition~3.4]{Beker2025} holds with base set $B^{(j)}$, averaging Bohr sets $B',B''$, first endpoint $2\widetilde A_j$, and second endpoint $T$.

The first endpoint has density $\asymp\alpha$ in $B^{(j)}$, while the second has logarithmic density $O(k^3L^2)$.  Substitution into the rank formula of \cite[Proposition~3.4]{Beker2025} gives
\[
|\Delta_G|
\ll k^2\,(k^4L^2)^2\,(k^3L^2)^2\,L\,(k^3L^2)
\ll k^{19}L^{11}.
\]
The width formula in \cite[Proposition~3.4]{Beker2025} gives logarithmic loss $O(k\Lcal(\alpha/d)+k^3L^2)$.  Proposition~3.4 returns a relative increment $1+c\eps_G$ on its designated first endpoint.  Since $\eps_G\gg\gamma$, Lemma~\ref{lem:local-normalisation} converts this to an increment of size $1+c'\eps_G$ relative to $\alpha$, and the inverse-dilation argument from the $D$-branch turns the increment of $2\widetilde A_j$ into one for a translate of $A$ without changing rank, width, or regularity.  This is the $G$-type alternative.

It remains to treat Proposition~\ref{prop:fixed-moment}\ref{it:C}.  Lemma~\ref{lem:gram-tail} applies.  Let $H_C\ge1$ be a common upper bound for the logarithmic densities, relative to the appropriate $B^{(i)}$, of the first-coordinate carriers produced by Lemma~\ref{lem:gram-projection}.  From \eqref{eq:Eplus-mass}, together with $\eta\asymp k^{-2}$, $P=O(k^2L)$ and $\rho_e\asymp\alpha$, we may take
\[
H_C=O(k^2L\Lambda).
\]  Here
\[
\tau=\frac{\eta^2}{4}=\frac{1}{1024m^2},
\qquad
\frac{\tau}{8}=\frac{1}{8192m^2}>7\delta,
\qquad
\frac{\tau}{4}>15\delta.
\]
Since $\gamma=2^{-10}\delta$, these inequalities imply in particular that $\tau/8>\delta+4\gamma$.  Hence Lemma~\ref{lem:local-normalisation} applies to both centred effects with an absolute proportion of their original strength.

In the low-degree outcome of Lemma~\ref{lem:gram-projection}, the carrier $S$ satisfies
$d_e(x)\le(1-\tau/8)\rho_e$ for all $x\in S$.  Lemma~\ref{lem:local-normalisation}, with $T=A_j$ and $u=\tau/8$, therefore gives a correlation defect $\eps_C\asymp\tau\asymp k^{-4}$.  As above,
$B^{(i)}=(B^{(j)})_{\lambda_i/\lambda_j}$ is regular and
$\lambda_i/\lambda_j\le c\eps_C\alpha_0/\widetilde d$ after the absolute constant in $\lambda$ is chosen sufficiently small.  Thus \cite[Corollary~3.5]{Beker2025} applies with endpoint-density parameter $\alpha_0$ and gives additional rank
\[
\ll k^{16}(k^2L\Lambda)^2\Lambda^2L^2
\ll k^{20}L^4\Lambda^4
\le k^{20}L^3\Lambda^5.
\]
The high-degree outcome is handled in the same way.  Taking $S=H$, its row density is larger than $2\rho_e$, so Lemma~\ref{lem:local-normalisation} may be applied with $u=1$ and $T=A_j$, giving a positive correlation defect of constant size.  The carrier still has logarithmic density at most $H_C+O(1)$, so the rank and width costs are smaller than the displayed bound.  In both the low- and high-degree outcomes, the increment returned on either dense endpoint is converted back to a relative increment over $\alpha$ by the final assertion of Lemma~\ref{lem:local-normalisation}; if necessary, inverse dilation returns the increment from $2\widetilde A_j$ to $A'$ as in the $D$-branch.

Consider finally the positive-rectangle outcome.  We have $C\subseteq A_i$, $T\subseteq A_j$,
\[
\dens_{A_i}(C)\ge e/2,
\qquad
\dens_{A_j}(T)\ge\tau\rho_e^2,
\]
and edge density at least $(1+\tau/4)\rho_e$ on $C\times T$.  The second inequality implies that $T$ has Bohr-relative density at least $c\tau\alpha^3\gg k^{-4}\alpha^3$, and hence logarithmic density $O(\Lambda)$.  The first carrier has logarithmic density $H_C+O(1)$.  Since $\tau/4>\delta+4\gamma$, Lemma~\ref{lem:local-normalisation} gives
\[
\langle\mu_{2\widetilde A_j}\circ\mu_T,\mu_C\rangle
\ge(1+c\tau)\mu(B^{(j)})^{-1}.
\]
H\"older lifting as above gives an $L^p(\mu_{B^{(i)}})$ anomaly.  Choose the integer
\[
p_C=\left\lceil C_1\tau^{-1}
\bigl(H_C+\Lcal(\tau)\bigr)\right\rceil
=O(k^6L\Lambda).
\]
The term $\tau^{-1}H_C$ preserves an absolute proportion of the $\tau$-sized anomaly after H\"older lifting, while the term $\tau^{-1}\Lcal(\tau)$ guarantees the moment condition $p_C\ge C\eps_C^{-1}\Lcal(\eps_C)$ for $\eps_C\asymp\tau$.  Choosing regular $B',B''$ exactly as in the $G$-branch and applying \cite[Lemmas~B.11 and~A.2]{Beker2025}, we obtain the hypothesis of \cite[Proposition~3.4]{Beker2025} with $\eps_C\asymp\tau$, first endpoint $2\widetilde A_j$, and second endpoint $T$.  Using $\Lcal(\alpha_1)=O(L)$ for the dense first endpoint and $\Lcal(\alpha_2)=O(\Lambda)$ for $T$, the rank formula in \cite[Proposition~3.4]{Beker2025} gives
\[
|\Delta_{C,+}|
\ll
k^8\,(k^6L\Lambda)^2\,\Lambda^2\,L\,\Lambda
\ll k^{20}L^3\Lambda^5.
\]
As before, the designated-side conclusion first gives a relative increment $1+c\eps_C$ over the density of $2\widetilde A_j$ in $B^{(j)}$.  Since $\eps_C\gg\gamma$, Lemma~\ref{lem:local-normalisation} converts this to an increment of order $\eps_C$ over $\alpha$, and inverse dilation returns it to a translate of $A$.  This completes the $C$-type alternative.  The pre-graph edge-discrepancy case in Beker's proof has relative gain $\asymp k^{-4}$ and a smaller rank cost, so it is absorbed into the same budget class.

For completeness, the lower bound
$\lambda_i/\lambda_j\ge(\lambda/2)^{j-i}$ shows that any use of Corollary~3.5 loses at most $O(k\Lcal(\alpha/d)+k\log k)$ in logarithmic width, after including its polynomial factors in $\widetilde d$, $\eps^{-1}$ and the additional rank.  For Proposition~3.4, the parameter denoted by $\rho$ in the backend is the width of the inner averaging Bohr set $B''$, not the width of $B^{(j)}$.  Since $B''$ is obtained from $B^{(i)}$ by two regular dilations of factors $\asymp1/\widetilde d$, this distinction contributes only $O(\log(2\widetilde d))$.  The backend then gives
\[
\rho'\gg\rho(B'')\eps(\alpha_1\alpha_2)^{1/2}/(\widetilde d^{3}d').
\]
The second-endpoint entropy is $O(\Lambda)$ in the centred branch and $O(k^3L^2)$ in the $G$-branch.  Together with the nested-scale loss and the polynomial factors in the ranks, this gives precisely the width bounds stated in the proposition.  This completes the proof.
\end{proof}

\section{Iteration and proof of the main theorem}\label{sec:iteration}

The local alternatives in Proposition~\ref{prop:local} have different costs and different density gains.  The key bookkeeping point is to weight every occurrence by its actual gain.

\begin{lemma}[Branch budget]\label{lem:budget}
Let $\alpha_0,\ldots,\alpha_t\le1$ be the densities in a density-increment iteration. Suppose the branch used at step $s$ is $b_s$ and
\[
\alpha_s\ge(1+c g_{b_s})\alpha_{s-1},\qquad 0<g_{b_s}\le1.
\]
Put $L_0=\log(2/\alpha_0)$ and $g_*=\min_b g_b$.  Then
\begin{equation}\label{eq:density-budget}
\sum_{s=1}^t g_{b_s}\ll L_0,
\end{equation}
and in particular $t\ll L_0/g_*$.  If branch $b$ introduces at most $D_b$ new auxiliary frequencies and has logarithmic width cost at most $W_b$, then
\begin{equation}\label{eq:sum-Delta}
\sum_{s=1}^t|\Delta_s|
\ll L_0\max_b\frac{D_b}{g_b},
\end{equation}
while
\begin{equation}\label{eq:width-budget}
\log\frac{\rho_0}{\rho_t}
\ll L_0\max_b\frac{W_b}{g_b}.
\end{equation}
\end{lemma}

\begin{proof}
For $0<x\le1$, $\log(1+cx)\ge c'x$ for an absolute $c'>0$.  Taking logarithms of the product of the density increments and using $\alpha_t\le1$ gives \eqref{eq:density-budget}.  Then
\[
\sum_s|\Delta_s|
\le\sum_b N_bD_b
=\sum_b(N_bg_b)\frac{D_b}{g_b}
\ll L_0\max_b\frac{D_b}{g_b},
\]
which is \eqref{eq:sum-Delta}.  The proof of \eqref{eq:width-budget} is identical.
\end{proof}

We also need the special frequency propagation already observed by Beker \cite[proof of Theorem~3.1]{Beker2025}.  If
\[
\Gamma_s\subseteq\Gamma_{s-1}\cup\{\gamma\psi_2^{\sigma_s}:\gamma\in\Gamma_{s-1}\}\cup\Delta_s,
\]
then, with $\Delta_0$ consisting of one trivial frequency,
\[
\Gamma_t\subseteq
\bigcup_{r=0}^t\{\gamma\psi_2^a:\gamma\in\Delta_r,\ |a|\le t-r\}.
\]
Thus
\begin{equation}\label{eq:orbit-rank}
d_t\le\sum_{r=0}^t(2(t-r)+1)|\Delta_r|.
\end{equation}

\begin{proposition}[Global rank and width]\label{prop:global-budget}
Let the local iteration start from a set of density $\alpha$, and write $L=\Lcal(\alpha)$ and $\Lambda=L+\log(2k)$.  Then the terminal Bohr set may be chosen with rank
\begin{equation}\label{eq:final-rank}
d_t\ll k^{29}L^5\Lambda^5+k^{25}L^{13}
\end{equation}
and width satisfying
\begin{equation}\label{eq:final-width}
\log\frac1{\rho_t}
\ll k^6L\Lambda+k^4L^3.
\end{equation}
\end{proposition}

\begin{proof}
Let $\alpha_s$ denote the density at step $s$ and put $L_s=\Lcal(\alpha_s)$ and $\Lambda_s=L_s+\log(2k)$.  Since every non-terminal step is a density increment, $\alpha_s\ge\alpha$ for all $s$, and therefore
\[
L_s\le L,\qquad \Lambda_s\le\Lambda.
\]
Consequently every local rank and branch-entropy bound from Proposition~\ref{prop:local} may be majorised uniformly by the same expression with the initial parameters $L$ and $\Lambda$.  We use this monotonicity throughout the following bookkeeping.

The four budget classes in Proposition~\ref{prop:local} have
\[
\begin{array}{c|c|c}
 b & g_b & D_b\\ \hline
 U & k^{-5} & 0\\
 D & k^{-2} & O(k^{12}L^4\Lambda^4)\\
 C & k^{-4} & O(k^{20}L^3\Lambda^5)\\
 G & k^{-1} & O(k^{19}L^{11}).
\end{array}
\]
Consequently
\[
\max_b\frac{D_b}{g_b}
\ll k^{24}L^3\Lambda^5+k^{20}L^{11},
\qquad g_*=k^{-5}.
\]
Lemma~\ref{lem:budget} gives $t\ll k^5L$ and
\[
\sum_s|\Delta_s|
\ll k^{24}L^4\Lambda^5+k^{20}L^{12}.
\]
Using \eqref{eq:orbit-rank} proves \eqref{eq:final-rank}.

For the width, the $U,D,C$ classes have cost $O(k\Lambda)$ after the standard polynomial-rank bootstrap, while the $G$ class has cost $O(k^3L^2+k\Lambda)$.  Thus
\[
\frac{W_U}{g_U}\ll k^6\Lambda,
\quad
\frac{W_D}{g_D}\ll k^3\Lambda,
\quad
\frac{W_C}{g_C}\ll k^5\Lambda,
\quad
\frac{W_G}{g_G}\ll k^4L^2+k^2\Lambda.
\]
Equation \eqref{eq:width-budget} gives \eqref{eq:final-width}.

For completeness, the bootstrap just used is harmless: \eqref{eq:final-rank} shows that every rank encountered is polynomial in $k$ and $L$, and therefore
$\log(2d_s)=O(L+\log(2k))=O(\Lambda)$.  Hence every occurrence of $\Lcal(\alpha_s/d_s)$ in the local width formula is $O(\Lambda)$, as claimed.
\end{proof}

\begin{proof}[Proof of Theorem~\ref{thm:finite-group}]
Represent $G$ as the regular Bohr set $B_0=\Bohr(\{1\};2)$, where $1$ denotes the trivial character, so that the initial rank is one and the local expression $\Lcal(\alpha/d)$ is defined from the first step.  Starting from $(A_0,B_0)=(A,B_0)$, repeatedly apply Proposition~\ref{prop:local}.  If the counting alternative \ref{local-count} occurs, stop.  Otherwise one of the four increment branches occurs, and its relative density gain is at least a constant multiple of $k^{-5}$.  Since the density is at most one, Lemma~\ref{lem:budget} rules out more than $O(k^5L)$ consecutive non-counting steps.  Hence the process terminates after finitely many steps at a pair $(A_t,B_t)$ for which the counting alternative holds.  Since each $A_s$ is contained in a translate of $A_{s-1}$, the terminal set $A_t$ is contained in a translate of the original set, and translation invariance transfers every terminal configuration back to $A$.

The terminal count is therefore at least
\[
\exp\!\bigl(-C(k^2\log k)d_t\Lcal(\alpha_t/d_t)\bigr)\rho_t^{2kd_t}.
\]
By Proposition~\ref{prop:global-budget}, $d_t$ is polynomial in $k$ and $L$, whence $\Lcal(\alpha_t/d_t)=O(\Lambda)$.  The first exponential factor is bounded below by
\[
\exp\!\left(-C\bigl(k^{31}L^5\Lambda^7+k^{27}L^{13}\Lambda^2\bigr)\right),
\]
Since $L\ge\log 2$ and $k\ge2$, we have $\Lambda\ll k^5L$.  Therefore
\[
k^{31}L^5\Lambda^7\ll k^{36}L^6\Lambda^6,
\qquad
k^{27}L^{13}\Lambda^2\ll k^{32}L^{14}\Lambda,
\]
so this entire contribution is absorbed by the right-hand side of \eqref{eq:Phi}.  The factor $\rho_t^{2kd_t}$ contributes an exponent which, by \eqref{eq:final-rank} and \eqref{eq:final-width}, is at most a constant multiple of
\[
\begin{aligned}
&k\bigl(k^{29}L^5\Lambda^5+k^{25}L^{13}\bigr)
\bigl(k^6L\Lambda+k^4L^3\bigr)\\
&\qquad =
 k^{36}L^6\Lambda^6
+k^{34}L^8\Lambda^5
+k^{32}L^{14}\Lambda
+k^{30}L^{16}.
\end{aligned}
\]
This is exactly \eqref{eq:Phi}, and proves \eqref{eq:main-count}.

If $A$ contains no non-degenerate $k$-configuration, every counted tuple has $x_i=x_j$ for some $i\ne j$.  The proportion of such tuples is at most $\binom{k}{2}/|G|$.  Comparing this with \eqref{eq:main-count} and enlarging the absolute constant proves the final assertion.
\end{proof}

\begin{proof}[Proof of Corollary~\ref{cor:interval}]
Set $N'=2N+1$ and embed $[N]$ into the odd cyclic group $\mathbb Z/N'\mathbb Z$.  Exactly as in \cite[proof of Theorem~1.1]{Beker2025}, any non-degenerate $k$-configuration in the image of $A$ lifts to one in $A$, because all relations $x_i+x_j=2z$ occur without modular wraparound.  The image has density at least $\alpha/3$, so Theorem~\ref{thm:finite-group} applies after changing the absolute constant.
\end{proof}

\begin{proof}[Proof of Theorem~\ref{thm:erdos-moser}]
We first record the only place where the exponent from Theorem~\ref{thm:finite-group} enters the Erd\H{o}s--Moser argument.  Let $k$ be sufficiently large and suppose that $X\subseteq Y\subseteq\mathbb Z$ satisfy
\[
|Y|\le(1+k^{-29})|X|
\]
and that no $k$-element subset of $X$ is sum-free with respect to $Y$.  Beker's structural Proposition~4.2, whose proof is inherited from Sanders' argument, then provides $X'\subseteq X$ such that
\[
|X'|\gg k^{-29}|X|,
\qquad
|X'+X'|\ll k^{181}|X'|,
\qquad
(2\cdot X')\cap Y=\varnothing;
\]
see \cite[Proposition~4.2]{Beker2025}.  By the Pl\"unnecke--Ruzsa inequality \cite[Corollary~6.29]{TaoVu2006},
$|2X'-2X'|\ll k^{724}|X'|$.  Ruzsa's Freiman embedding lemma \cite{Ruzsa1994}, in the form used in \cite[Lemma~4.3]{Beker2025}, yields a subset $X''\subseteq X'$ with $|X''|\ge|X'|/2$, an odd integer
$N\ll k^{724}|X'|$, and a Freiman $2$-isomorphism from $X''$ into $\mathbb Z/N\mathbb Z$.  Its image has density $\gg k^{-724}$.

For this density, $L,\Lambda=O(\log k)$, so Theorem~\ref{thm:finite-group} applies whenever
\[
N\ge \exp(k^{36+o(1)}).
\]
Since $N\ge|X''|\gg k^{-29}|X|$, the latter is forced by
$|X|\ge\exp(k^{36+o(1)})$.  We then obtain a non-degenerate $k$-configuration in the cyclic image of $X''$.  A Freiman $2$-isomorphism preserves the relations defining arithmetic means, so this configuration pulls back to distinct generators $x_1,\ldots,x_k\in X''$ whose pairwise sums all belong to $2\cdot X''\subseteq2\cdot X'$.  The last set is disjoint from $Y$, and hence $\{x_1,\ldots,x_k\}$ is sum-free with respect to $Y$, a contradiction.  Thus Beker's Proposition~4.1 is valid with its threshold $\exp(k^{68+o(1)})$ replaced by $\exp(k^{36+o(1)})$; compare \cite[Proposition~4.1 and its proof]{Beker2025}.

The deduction of the Erd\H{o}s--Moser lower bound from this structural proposition is exactly Sanders' iteration \cite[Proposition~2.7 and Theorem~1.2]{Sanders2021}, also recalled in \cite[Section~4]{Beker2025}.  In that deduction a threshold $\exp(k^{\beta+o(1)})$ permits every exponent $c<1/\beta$.  Taking $\beta=36$ proves the theorem.
\end{proof}

\appendix
\section{Technical verification and external inputs}

This appendix records the precise external ingredients used in the main proof and explains why their hypotheses match the branch-sensitive routing.  No new quantitative optimisation is performed here.

\subsection{The two local density-increment backends}

We record the hypotheses of the two results from Beker that are used in Proposition~\ref{prop:local}.  This is useful because the ambient Bohr set, the averaging Bohr sets, and the endpoint densities play different roles.

First, \cite[Corollary~3.5]{Beker2025} states the following.  Let $B=\Bohr(\Gamma;\rho)$ be a regular Bohr set of rank $d_0$, and let $A_1,A_2\subseteq B$ have relative densities at least $\alpha_0>0$.  Let $\eps\in(0,1]$, and let
$B_{\lambda'}$ be a regular dilate with
\[
0<\lambda'\le c\eps\alpha_0/d_0.
\]
If $C\subseteq B_{\lambda'}$ has relative density $\gamma_C>0$ and
\[
\left|\langle\mu_{A_1}\circ\mu_{A_2},\mu_C\rangle-\mu(B)^{-1}\right|
\ge\eps\mu(B)^{-1},
\]
then one of $A_1,A_2$ gains density on a regular Bohr set whose additional rank is
\[
O\!\left(
\eps^{-4}
\Lcal(\min(\eps,\gamma_C))^2
\Lcal(\eps\alpha_0)^2
\Lcal(\alpha_0)^2
\right).
\]
The output width is bounded below by the formula in \cite[Corollary~3.5]{Beker2025}.  In the $D$-type and degree-based centred branches we take $B=B^{(j)}$, $B_{\lambda'}=B^{(i)}$, $d_0=\widetilde d\le2d$, and $\alpha_0=(1-\gamma)\alpha$.  Both endpoints $A_j$ and $2\widetilde A_j$ have density at least $\alpha_0$, and the carrier-scale condition follows from
$\lambda_i/\lambda_j\le\lambda\ll\eps\alpha_0/\widetilde d$.  Thus the hypotheses match literally after this constant-factor relabelling of the endpoint-density parameter.

Second, \cite[Proposition~3.4]{Beker2025} has a different geometry.  Let $B^{(1)}$ and $B^{(2)}=\Bohr(\Gamma_0;\rho_2)$ be regular Bohr sets of the same rank $d_0$ with
\[
B^{(2)}\subseteq (B^{(1)})_{c/d_0}.
\]
Let $A_1,A_2$ be subsets of translates of a common non-empty set $B_0$, with respective densities $\alpha_1,\alpha_2>0$.  If, for some $x\in G$, $\eps\in(0,1]$, and an integer
\[
p\ge C\eps^{-1}\Lcal(\eps),
\]
one has
\[
\|\mu_{A_1}\circ\mu_{A_2}\|_{L^p(\mu_{B^{(1)}}*\mu_{B^{(2)}+x})}
\ge(1+\eps)\mu(B_0)^{-1},
\]
then the designated first endpoint $A_1$ gains density on a regular Bohr set.  Its additional rank satisfies
\[
O\!\left(
\eps^{-2}p^2
\Lcal(\eps\sqrt{\alpha_1\alpha_2})^2
\Lcal(\alpha_1)\Lcal(\alpha_2)
\right),
\]
and its width is at least a constant multiple of
$\rho_2\eps(\alpha_1\alpha_2)^{1/2}/(d_0^3d')$, where $d'$ denotes the additional rank.  The proposition allows the translation $x$ explicitly; equivalently, it may be absorbed into the second endpoint as observed in the proof of \cite[Proposition~3.4]{Beker2025}.

In the positive $G$- and $C$-rectangle branches we take $B_0=B^{(j)}$ and $A_1=2\widetilde A_j$, while $A_2=T$ may be sparse.  We choose regular averaging sets
$B'=B^{(i)}_{\lambda'}$ and $B''=(B')_{\lambda''}$ with
$\lambda',\lambda''\in[c/(2\widetilde d),c/\widetilde d]$, using \cite[Lemma~B.10]{Beker2025}.  Hence $B''\subseteq(B')_{c/\widetilde d}$ and both sets have rank $\widetilde d$.  Lemma~B.11 and Lemma~A.2 of \cite{Beker2025} provide the required transfer from $L^p(\mu_{B^{(i)}})$ to $L^p(\mu_{B'}*\mu_{B''+x})$.  The choices of $p_G$ and $p_C$ in the proof of Proposition~\ref{prop:local} include the extra terms $\theta^{-1}\Lcal(\theta)$ and $\tau^{-1}\Lcal(\tau)$ precisely to ensure the moment lower bound in Proposition~3.4.

\subsection{Compatibility of normalisation, dilation, and the nested Bohr scale}

The local-uniformisation step gives
\[
|\alpha_{ij}-\alpha|\le\frac{\delta}{4}\alpha,
\qquad
|\mu_{B^{(j)}}(2\widetilde A_j)-\alpha|\le\gamma\alpha.
\]
Lemma~\ref{lem:local-normalisation} records the exact conversion between a rectangle-density anomaly and the normalised difference-convolution used by the two Beker backends.  In particular, it is enough to assume $u\ge\delta+4\gamma$.  Since
\[
\frac{\tau}{8}=\frac{1}{8192m^2}>7\delta>\delta+4\gamma,
\]
the low-degree centred branch satisfies this condition.  The same lemma also converts an increment relative to the local endpoint density back to an increment relative to the iterative density $\alpha$ whenever the backend gain dominates $\gamma$; this holds in all $D$-, $C$-, and $G$-branches.

Beker's nested construction starts from
\[
\widetilde B=B\cap(2\cdot B),
\qquad
\widetilde\Gamma=\Gamma\cup\{\xi\psi_2^{-1}:\xi\in\Gamma\},
\qquad
\widetilde d\le2d,
\]
and uses a base ratio
\[
\lambda\asymp\frac{\alpha}{d}k^{-5}.
\]
For every pair $i<j$ the regular nested scales satisfy
\[
(\lambda/2)^{j-i}\le\lambda_i/\lambda_j\le\lambda.
\]
The upper bound gives the applicability condition for Corollary~3.5, because the weakest defect routed through that result has size $\eps\asymp k^{-4}$, while $\lambda$ is smaller by an additional factor of order $k^{-1}$.  The lower bound controls the accumulated width loss and contributes only
$O(k\Lcal(\alpha/d)+k\log k)$.

Finally, multiplication by $2$ is an automorphism because $|G|$ is odd.  If a backend returns an increment on $2\widetilde A_j$, the inverse dilation $x\mapsto x/2$ sends
$\Bohr(\Gamma';\rho')$ to $\Bohr(\Gamma'\psi_2;\rho')$.  Thus rank and width are unchanged, and regularity is preserved because the automorphism preserves the cardinalities of all physical dilates.  This justifies the pullback step and the frequency propagation by powers of $\psi_2$ used later in the iteration.

By contrast, the sparse second endpoints in the positive rectangle branches may have densities as small as $k^{-4}\alpha^3$ in the centred case and exponentially small densities in the $G$-case.  Routing these branches through Corollary~3.5 would force the reference scale to depend on that sparse density.  Proposition~3.4 avoids this: the sparse endpoint density enters its rank and width formulas, but not a carrier-dilate condition.  This is the structural reason for the routing used in Proposition~\ref{prop:local}.

\subsection{Sparse rectangle extraction}

We use \cite[Corollary~2.3]{FHHK2024} in the following form.  If a bipartite graph $f\colon X\times Y\to\{0,1\}$ of density $\rho$ has a $K_{r,p}$ count at least
$(1+\eps)^{rp}\rho^{rp}$, then there are $S\subseteq X$ and $T\subseteq Y$ with
\[
\frac{|S||T|}{|X||Y|}\gg \eps\rho^{rp+1}
\]
and edge density at least $(1+\eps/2)\rho$ on $S\times T$.  In the $G$-branch, $r\le2(k-1)$, $p=P=O(k^2L)$ and $\eps\asymp k^{-1}$, whence each side has logarithmic density $O(k^3L^2)$.

\subsection{The Erd\H{o}s--Moser reduction}

The final application uses no new ingredient beyond the improved finite-group threshold.  Beker's Section~4 combines a structural proposition inherited from the Sudakov--Szemer\'edi--Vu/Shao/Sanders route with Pl\"unnecke--Ruzsa and Ruzsa's Freiman $2$-embedding.  It produces a cyclic model of density $\gg k^{-724}$, so any configuration theorem with threshold $\exp(k^{\beta+o(1)})$ yields the corresponding Erd\H{o}s--Moser exponent $c<1/\beta$ through Sanders' iteration.  Substituting $\beta=36$ gives Theorem~\ref{thm:erdos-moser}.

\section*{Declaration on the use of artificial intelligence}

OpenAI ChatGPT (GPT-5.6 Sol, accessed through a ChatGPT Plus subscription) was used during the preparation of this manuscript to assist with literature organisation, LaTeX drafting and editing, language polishing, and iterative checking of proof details and quantitative bookkeeping.  Outputs from the tool were not treated as authoritative sources and were checked against the cited literature where relevant.  Responsibility for the mathematical content of the manuscript, including all proofs, citations, and conclusions, remains entirely with the human author .

\end{document}